%% file: syncaver.tex
\documentclass{article}

\input{macros}

\begin{document}

\title{Synchronization analysis of coupled planar oscillators
by averaging}
\author{S. Emre Tuna\\
{\em \small{Middle East Technical University, Ankara, Turkey}}\\
{\tt \small tuna@eee.metu.edu.tr}}
\maketitle

\begin{abstract}
Sufficient conditions for synchronization of coupled Lienard-type
oscillators are investigated via averaging technique. Coupling
considered here is pairwise, unidirectional, and described by a
nonlinear function (whose graph resides in the first and third
quadrants) of some projection of the relative distance (between
the states of the pair being coupled) vector. Under the assumption
that the interconnection topology defines a connected graph, it is
shown that the solutions of oscillators can be made converge
arbitrarily close to each other, while let initially be
arbitrarily far apart, provided that the frequency of oscillations
is large enough and the initial phases of oscillators all lie in
an open semicircle. It is also shown that (almost) synchronized
oscillations always take place at some fixed magnitude independent
of the initial conditions. Similar results are generated for
nonlinearly-coupled harmonic oscillators.
\end{abstract}

\section{Introduction}
Synchronization in coupled dynamical systems has been a common
ground of investigation for researchers from different
disciplines. Most of the work in the area studies the case where
the coupling between individual systems is linear; see, for
instance, \cite{wu95,pecora98,pogromsky02,wu05,belykh06,liu08}.
Nonlinear coupling is also of interest since certain phenomena
cannot be properly modelled by linear coupling. A particular
system exemplifying nonlinear coupling that attracted much
attention is Kuramoto model and its like
\cite{bonilla98,pikovsky09}. Among more general results allowing
nonlinear coupling are \cite{arcak07,stan07} where passivity
theory is employed to obtain sufficient conditions for
synchronization under certain symmetry or balancedness assumptions
on the coupling graph.

In this paper we study the synchronization behavior of an array of
nonlinear planar oscillators. We let the individual oscillators
share identical dynamics and the coupling between them be
nonlinear. We consider Lienard-type oscillators that have been
much studied due to their close relation to real-life systems and
applications \cite{elrady04,santos08,slight08}. A particular
example is van der Pol oscillator \cite{guckenheimer97,murali93}.

What we investigate here is the {\em relation between frequency of
oscillations and synchronization} of the oscillators forming the
array. The array is formed such that some of the oscillators are
coupled to some others via a nonlinear function. Coupling
considered is of partial-state nature. That is, if an oscillator
affects the dynamics of another, the associated coupling term is
not a function of the state vector of the oscillator that is
affecting, but only of some projection of that vector. We make no
symmetry nor balancedness assumption on the coupling graph.

Our finding in the paper is roughly that (almost) synchronization
occurs among the oscillators (at some magnitude independent of the
initial conditions, the coupling, and the frequency of
oscillations $\omega$) if the following conditions hold: (a) the
frequency of oscillations is high, (b) there is at least one
oscillator that directly or indirectly affects all others, and (c)
the initial phases of oscillators, when each is represented by a
point on the unit circle, all lie in an open semicircle. More
formally, what we show is that if the coupling graph is connected
and the initial phases of oscillators lie in an open semicircle,
then the solutions of oscillators can be made converge arbitrarily
close to each other, while initially being arbitrarily far from
one another, by choosing large enough $\omega$. Incidentally, as
sort of a byproduct of our analysis for nonlinear oscillators, we
also generate a similar result for an array of nonlinearly-coupled
harmonic oscillators. We show that harmonic oscillators (almost)
synchronize provided that the coupling graph is connected and the
frequency of oscillations is high. Different from nonlinear
oscillators, the initial phases of harmonic oscillators do not
have any effect on synchronization, at least when $\omega$ is
sufficiently high. Also, again unlike nonlinear oscillators,
synchronized oscillations can take place at any magnitude,
depending on the initial conditions.

In establishing our results we use tools from averaging theory
\cite{arnold88}. Our three-step approach is as follows. We first
apply a time-varying change of coordinates to the array, which
keeps the relative distances between the states of oscillators
intact (Section~\ref{sec:coc}). This change of coordinates renders
the array periodically time-varying. Then we take the time-average
of this new array and show, for the average array, that the
oscillators synchronize if their initial phases lie in an open
semicircle and the coupling graph is connected
(Section~\ref{sec:average}). Finally we work out what that result
means for the original array. Namely, we show that the higher the
frequency of oscillations the more the original array behaves like
its average (Section~\ref{sec:sync}).

\section{Preliminaries}\label{sec:pre}
Let $\Real_{\geq 0}$ denote the set of nonnegative real numbers.
Let $|\!\cdot\!|$ denote Euclidean norm. A function
$\alpha:\Real_{\geq 0}\to \Real_{\geq 0}$ is said to belong to
class-$\K$ $(\alpha\in\K)$ if it is continuous, zero at zero, and
strictly increasing. It is said to belong to class-$\Kinf$ if it
is also unbounded. Given a closed set ${\mathcal
S}\subset\Real^{n}$ and a point $x\in\Real^{n}$, $|x|_{\mathcal
S}$ denotes the (Euclidean) distance from $x$ to ${\mathcal S}$.
Convex hull of ${\mathcal S}$ is denoted by $\co{\mathcal S}$.
Number of elements of a finite set $\I$ is denoted by $\#\I$.

A ({\em directed}) {\em graph} is a pair $(\N,\,\setE)$ where $\N$
is a nonempty finite set (of {\em nodes}) and $\setE$ is a finite
collection of ordered pairs ({\em edges}) $(n_{i},\,n_{j})$ with
$n_{i},\,n_{j}\in\N$. A {\em directed path} from $n_{1}$ to
$n_{l}$ is a sequence of nodes $(n_{1},\,\ldots,\,n_{l})$ such
that $(n_{i},\,n_{i+1})$ is an edge for
$i\in\{1,\,\ldots,\,l-1\}$. A graph is {\em connected} if it has a
node to which there exists a directed path from every other
node.\footnote{This is another way of saying that the graph
contains a directed spanning tree.}

A set of functions $\{\gamma_{ij}:\Real\to\Real\}$, where
$i,\,j=1,\,\ldots,\,m$ with $i\neq j$, describes (is) an {\em
interconnection} if the following hold for all $i,\,j$ and all
$s\in\Real$:
\begin{itemize}
\item[\bf (G1)] $\gamma_{ij}(0)=0$ and $s\gamma_{ij}(s)\geq 0$.
\vspace{-0.05in} \item[\bf (G2)] Either $\gamma_{ij}(s)\equiv 0$
or there exists $\alpha\in\K$ such that $|\gamma_{ij}(s)|\geq
\alpha(|s|)$.
\end{itemize}
To mean $\gamma_{ij}(s)\equiv 0$ we write $\gamma_{ij}=0$.
Otherwise we write $\gamma_{ij}\neq 0$. The graph of
interconnection $\{\gamma_{ij}\}$ is pair $(\N,\,\setE)$, where
$\N=\{n_{1},\,\ldots,\,n_{m}\}$ and $\setE$ is such that
$(n_{i},\,n_{j})\in\setE$ iff $\gamma_{ij}\neq 0$. An
interconnection is said to be {\em connected} when its graph is
connected. To give an example, consider some set of functions
$\G:=\{\gamma_{ij}: i,\,j=1,\,\ldots,\,4\}$.  Let
$\gamma_{13},\,\gamma_{23},\,\gamma_{24},\,\gamma_{32}$ be as in
Fig.~1 and the remaining functions be zero. Note that each
$\gamma_{ij}$ satisfy conditions (G1) and (G2). Therefore set $\G$
describes an interconnection. To determine whether $\G$ is
connected or not we examine its graph, see Fig.~2. Since there
exists a path to node $n_{4}$ from every other node, we deduce
that the graph (hence interconnection $\G$) is connected.

\begin{figure}[h]
\begin{center}
\includegraphics[scale=0.6]{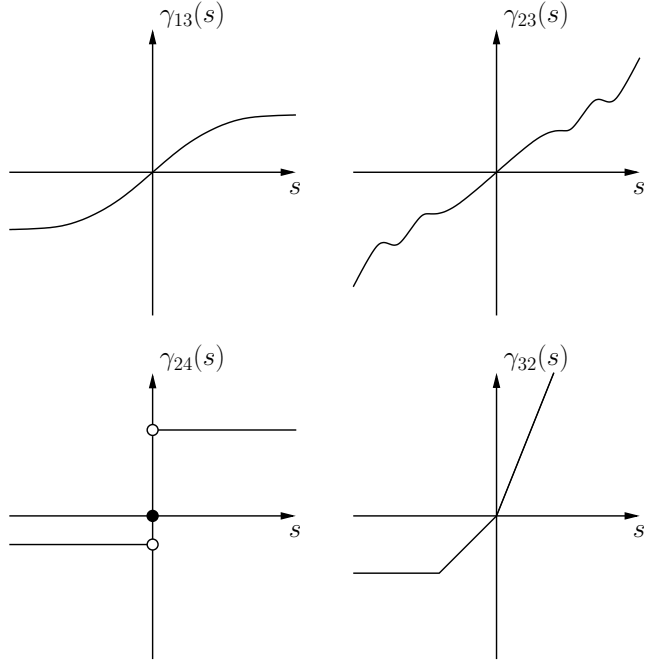}
\caption{Some examples of coupling functions.}
\end{center}
\end{figure}\label{fig:gammaij}

\begin{figure}[h]
\begin{center}
\includegraphics[scale=0.6]{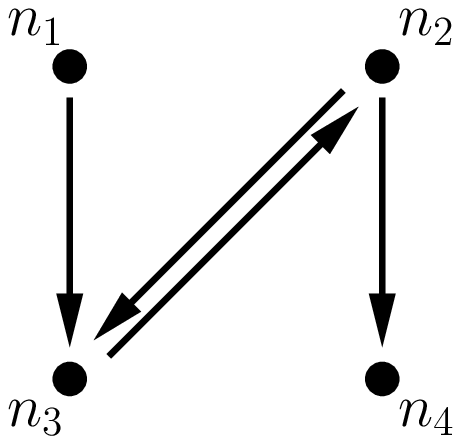}
\caption{Graph of interconnection $\G$.}
\end{center}
\end{figure}

Consider two half lines on $\Real^{2}$ (originating from the
origin) such that the (smaller) angle between them is strictly
less than $\pi$. Then their convex hull $\C$ is called a {\em
cone}, see Fig.~3. {\em Angle} of cone $\C$, denoted by
$\angle\C$, is the angle between the two half lines. Note that
$\angle\C\in[0,\,\pi)$ and a half line is a cone with zero angle.
Given $r>0$, a set $\W_{r}\subset\Real^{2}$ is called an {\em
$r$-wedge} if it can be written as
$\W_{r}=\co(\C\cap\{x\in\Real^{2}:r_{1}\leq |x|\leq r_{2}\})$ for
some cone $\C$ and $0<r_{1}\leq r\leq r_{2}$. See Fig.~4.

\begin{figure}[h]
\begin{center}
\includegraphics[scale=0.6]{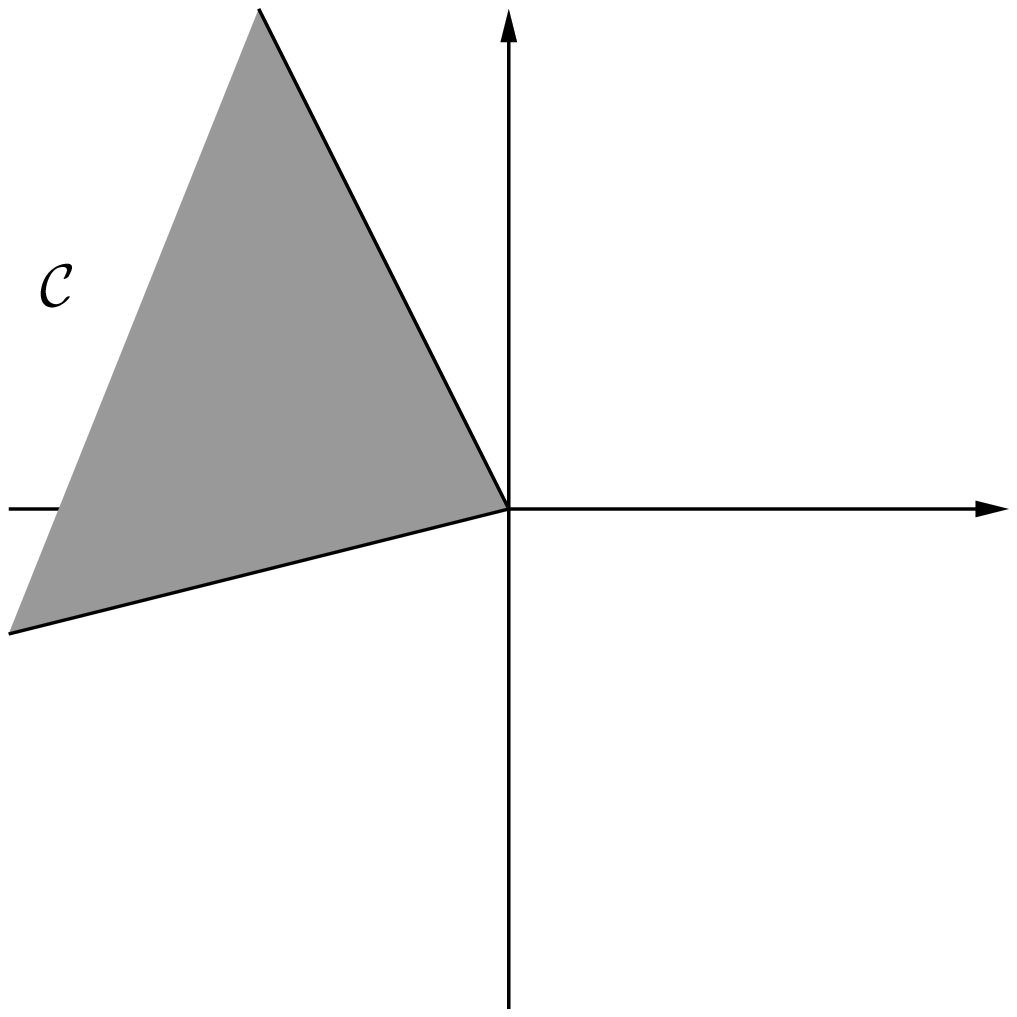}
\caption{A cone $\C$.}
\end{center}
\end{figure}\label{fig:cone}

\begin{figure}[h]
\begin{center}
\includegraphics[scale=0.6]{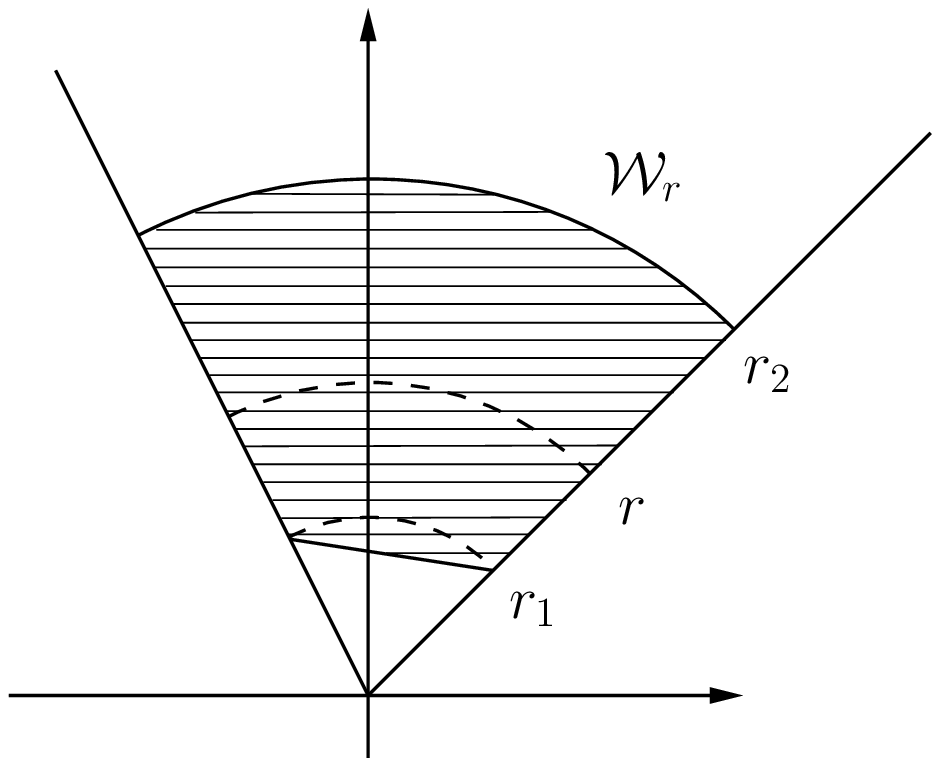}
\caption{An $r$-wedge $\W_{r}$.}
\end{center}
\end{figure}\label{fig:wedger}

\section{Problem statement}\label{sec:ps}
A general model for a planar oscillator is given by {\em Lienard's
equation}
\begin{eqnarray}\label{eqn:lienard}
{\ddot q}+f(q){\dot q}+g(q)=0\,.
\end{eqnarray}
Sufficient conditions have been established on functions $f$ and
$g$ in order for system~\eqref{eqn:lienard} to admit a unique,
stable limit cycle encircling the origin of the phase plane; see,
for instance, \cite[Thm.~1]{elrady04}. In this paper we adopt
those general conditions on function $f$, which are stated later
in the section. Regarding function $g$, we assume linearity.
Namely, we study the case $g(q)=\omega^{2}q$, where $\omega$ is
some constant. Although this assumption is restrictive compared to
what is assumed in \cite[Thm.~1]{elrady04}, it nevertheless still
allows us cover some physically important cases. Most famous
example that falls into the class of systems being studied here
would be van der Pol oscillator \cite{khalil96}. In this paper we
search for sufficient conditions that yield synchronization of a
number of coupled Lienard oscillators. Below we give the precise
description of the problem.

Consider the following array of coupled Lienard oscillators
\begin{subeqnarray}\label{eqn:lienosc}
\dot{q}_{i}&=&\omega p_{i}\\
\dot{p}_{i}&=&-\omega q_{i}-f(q_{i})p_{i}+\sum_{j\neq
i}\gamma_{ij}(p_{j}-p_{i})\,, \qquad i=1,\,\ldots,\,m
\end{subeqnarray}
where $\omega>0$ is the {\em frequency of oscillations} and
$\{\gamma_{ij}\}$ is a connected interconnection. Let
$\xi_{i}\in\Real^{2}$ denote the {\em state} of $i$th oscillator,
i.e., $\xi_{i}=[q_{i}\ p_{i}]^{T}$. Sometimes we choose to handle
this array~\eqref{eqn:lienosc} of $m$ planar oscillators as a
single system in $\Real^{2m}$. If we let
$\xi:=[\xi_{1}^{T}\,\ldots\,\xi_{m}^{T}]^{T}$ and
\begin{eqnarray*}
\ell(\xi,\,\omega):= \left[\begin{array}{c} \left[\begin{array}{l}
\omega p_{1}\\
-\omega q_{1}-f(q_{1})p_{1}+\sum\gamma_{1j}(p_{j}-p_{1})
\end{array}\right]\\
\vdots\\
\left[\begin{array}{l}
\omega p_{m}\\
-\omega q_{m}-f(q_{m})p_{m}+\sum\gamma_{mj}(p_{j}-p_{m})
\end{array}\right]
\end{array}\right]
\end{eqnarray*}
then array~\eqref{eqn:lienosc} defines the below system
\begin{eqnarray}\label{eqn:ell}
\dot\xi=\ell(\xi,\,\omega)\,.
\end{eqnarray}
We assume throughout the paper that $\gamma_{ij}$ and
$f:\Real\to\Real$ are locally Lipschitz. Letting
$F(s):=\int_{0}^{s}f(\sigma)d\sigma$, we further assume the
following.
\begin{itemize}
\item[\bf (L1)] $f$ is an even function. \item[\bf (L2)]
$F(s_{0})=0$ for some $s_{0}>0$; $F$ is negative on $(0,\,s_{0})$;
$F$ is positive, nondecreasing, and unbounded on
$(s_{0},\,\infty)$.
\end{itemize}
In this paper the question we ask ourselves is the following.
\begin{center}
{\em What is the effect of $\omega$ on the synchronization
behavior of array ~\eqref{eqn:lienosc}?}
\end{center}
Our approach to the problem is as follows. We first apply a
norm-preserving, time-varying change of coordinates to
array~\eqref{eqn:lienosc} and obtain a new array whose righthand
side is periodic in time. Since this coordinate change does not
alter the relative distances between the states of oscillators, it
is safe to look at the new array to understand the synchronization
behavior of the original one. Hence we focus on the new array.
Although the righthand side of the new array does not look any
pleasanter than the original one, it nevertheless has an {\em
average} since it is periodic in time. Once this average is
computed we realize that it is very simple to understand from
synchronization point of view. From then on we start going
backwards. Using averaging theory, we establish what our findings
on the average array imply, first for the time-varying array and
eventually for the original array.

\section{Change of coordinates}\label{sec:coc}

We define $S(\omega)\in\Real^{2\times 2}$ and
$H,\,V\in\Real^{1\times 2}$ as
\begin{eqnarray*}
S(\omega):=
\left[\!\!
\begin{array}{rr}
0&\omega\\-\omega&0
\end{array}\!\!
\right]\,,\quad H:=[1\ \ 0]\,,\quad V:=[0\ \ 1] \,.
\end{eqnarray*}
Then we rewrite \eqref{eqn:lienosc} as
\begin{eqnarray}\label{eqn:lienosc2}
\dot{\xi}_{i}=S(\omega)\xi_{i}-f(H\xi_{i})V^{T}V\xi_{i}+V^{T}\sum_{j\neq
i}\gamma_{ij}(V(\xi_{j}-\xi_{i}))\,.
\end{eqnarray}
Let $x_{i}(t):=e^{-S(\omega)t}\xi_{i}(t)$ and $x:=[x_{1}^{T}\,
\ldots\,x_{m}^{T}]^{T}$. We can by \eqref{eqn:lienosc2} write
\begin{eqnarray}\label{eqn:lienoscc}
\dot{x}_{i}
&=&-f(He^{S(\omega)t}x_{i})e^{S(\omega)^{T}t}V^{T}Ve^{S(\omega)t}x_{i}\nonumber\\
&&\qquad +\sum_{j\neq i}e^{S(\omega)^{T}t}V^{T}\gamma_{ij}
(Ve^{S(\omega)t}(x_{j}-x_{i}))\nonumber\\
&=&-f([\cos\omega t\ \sin\omega t]x_{i})\left[\!\!\begin{array}{r}
-\sin\omega t\\ \cos\omega t
\end{array}\!\!\right]
[-\sin\omega t\ \cos\omega t]x_{i}\nonumber\\
&&\qquad+\sum_{j\neq i}\left[\!\!\begin{array}{r} -\sin\omega t\\
\cos\omega t
\end{array}\!\!\right]\gamma_{ij}
([-\sin\omega t\ \cos\omega t](x_{j}-x_{i}))\\
&=:&\ellc_{i}(x,\,\omega t)\,.\nonumber
\end{eqnarray}
Let
\begin{eqnarray*}
\ellc(x,\,\omega t):=\left[\begin{array}{c} \ellc_{1}(x,\,\omega
t)\\
\vdots\\
\ellc_{m}(x,\,\omega t)
\end{array}
\right]
\end{eqnarray*}
Then system in $\Real^{2m}$ corresponding to
array~\eqref{eqn:lienoscc} reads
\begin{eqnarray}\label{eqn:ellc}
\dot x=\ellc(x,\,\omega t)\,.
\end{eqnarray}

\begin{remark}\label{rem:arkabak}
Since the change of coordinates is realized via rotation matrix
$e^{-S(\omega)t}$, the relative distances are preserved, that is
$|x_{i}(t)-x_{j}(t)|=|\xi_{i}(t)-\xi_{j}(t)|$ for all $t$ and all
$i,\,j$. This means from synchronization point of view that the
behavior of array~\eqref{eqn:lienosc} will be inherited by
array~\eqref{eqn:lienoscc}.
\end{remark}
Exact analysis of \eqref{eqn:lienoscc} seems far from yielding.
Therefore we attempt to understand this system via its
approximation.

\section{Average array}\label{sec:average}

Observe that the righthand side of \eqref{eqn:lienoscc} is
periodic in time. Time-average functions $\bar
f:\Real^{2}\to\Real^{2}$ and
$\bar\gamma_{ij}:\Real^{2}\to\Real^{2}$ are given by
\begin{eqnarray*}
\bar f(v):=\frac{1}{2\pi}\int_{0}^{2\pi}
 f([\cos\varphi\ \sin\varphi]v)\left[\!\!
\begin{array}{r}
-\sin\varphi\\
\cos\varphi
\end{array}\!\!\right]
[-\sin\varphi\ \cos\varphi]v d\varphi
\end{eqnarray*}
and
\begin{eqnarray}\label{eqn:gammaijbar}
\bar{\gamma}_{ij}(v):=\frac{1}{2\pi}\int_{0}^{2\pi}
\left[\!\!\begin{array}{r} -\sin\varphi \\ \cos\varphi
\end{array}\!\!\right]\gamma_{ij}
([-\sin\varphi\ \cos\varphi]v)d\varphi\,.
\end{eqnarray}
Then the {\em average} array dynamics read
\begin{eqnarray}\label{eqn:lienosccbar}
{\dot\eta}_{i}=-\bar f(\eta_{i})+\sum_{j\neq i}
\bar\gamma_{ij}(\eta_{j}-\eta_{i})\,.
\end{eqnarray}
Let $\eta:=[\eta_{1}^{T}\,\ldots\,\eta_{m}^{T}]^{T}$ and
\begin{eqnarray*}
\ellcbar(\eta):=\left[\begin{array}{c} -\bar f(\eta_{1})+\sum
\bar\gamma_{1j}(\eta_{j}-\eta_{1})\\
\vdots\\
-\bar f(\eta_{m})+\sum \bar\gamma_{mj}(\eta_{j}-\eta_{m})
\end{array}
\right]
\end{eqnarray*}
Then system in $\Real^{2m}$ corresponding to
array~\eqref{eqn:lienosccbar} reads
\begin{eqnarray}\label{eqn:ellcbar}
\dot \eta=\ellcbar(\eta)\,.
\end{eqnarray}

\begin{remark}
Note that instead of array~\eqref{eqn:lienosc2} if we start with
the below array
\begin{eqnarray*}
\dot{\xi}_{i}=S(\omega)\xi_{i}-f(H\xi_{i})V^{T}V\xi_{i}+C^{T}\sum_{j\neq
i}\gamma_{ij}(C(\xi_{j}-\xi_{i}))\,,
\end{eqnarray*}
where $C\in\Real^{1\times 2}$ satisfies $CC^{T}=1$, we still reach
the same average array~\eqref{eqn:lienosccbar}. Therefore for the
analysis to follow the coupling need not be through {\em
velocities} as is the case in \eqref{eqn:lienosc}. Any projection
of the state $C\xi_{i}$ is as good as any other for coupling. For
instance, the results in this paper hold true for the below array
\begin{eqnarray*}
\dot{q}_{i}&=&\omega p_{i}+\sum_{j\neq
i}\gamma_{ij}(q_{j}-q_{i})\\
\dot{p}_{i}&=&-\omega q_{i}-f(q_{i})p_{i}\,.
\end{eqnarray*}
\end{remark}

Theory of perturbations \cite[Ch.~4~\S~17]{arnold88} tells us
that, starting from close initial conditions, the solution of a
system with a periodic righthand side and the solution of the
time-average approximate system stay close for a long time
provided that the period is small enough. Therefore
\eqref{eqn:lienosccbar} should tell us a great deal about the
behavior of \eqref{eqn:lienoscc} when $\omega\gg 1$. Understanding
\eqref{eqn:lienosccbar} requires understanding average functions
$\bar f$ and $\bar\gamma_{ij}$. Let us begin with the former.

\begin{lemma}\label{lem:keyf}
We have $\displaystyle \bar f(v)=\kappa(|v|)\frac{v}{|v|}$ where
$\kappa:\Real_{\geq 0}\to\Real$ is
\begin{eqnarray}\label{eqn:kappa}
\kappa(s):=\frac{2}{\pi}\int_{0}^{s}
f(\sigma)\sqrt{1-\frac{\sigma^{2}}{s^{2}}}d\sigma\,.
\end{eqnarray}
\end{lemma}

\begin{proof}
Given $v\in\Real^{2}$, let $r=|v|$ and $\theta\in[0,\,2\pi)$ be
such that $r[\cos\theta\ \sin\theta]^{T}=v$. Then, by using
standard trigonometric identities,
\begin{eqnarray*}
\bar f(v) &=&\frac{1}{2\pi}\int_{0}^{2\pi}
 f(r(\cos\varphi\cos\theta+\sin\varphi\sin\theta))\\
&& \qquad \times r(-\sin\varphi\cos\theta+\cos\varphi\sin\theta)
\left[\!\!
\begin{array}{r}
-\sin\varphi\\
\cos\varphi
\end{array}\!\!\right] d\varphi\\
&=&-\frac{r}{2\pi}\int_{0}^{2\pi} f(r(\cos(\varphi-\theta))
\sin(\varphi-\theta) \left[\!\!
\begin{array}{r}
-\sin\varphi\\
\cos\varphi
\end{array}\!\!\right] d\varphi\\
&=&-\frac{r}{2\pi}\int_{0}^{2\pi} f(r\cos\varphi) \sin\varphi
\left[\!\!
\begin{array}{r}
-\sin(\varphi+\theta)\\
\cos(\varphi+\theta)
\end{array}\!\!\right] d\varphi\\
&=&-\frac{r}{2\pi}\int_{0}^{2\pi} f(r\cos\varphi) \sin\varphi
\left[\!\!
\begin{array}{r}
-\sin\varphi\cos\theta-\cos\varphi\sin\theta\\
\cos\varphi\cos\theta-\sin\varphi\sin\theta
\end{array}\!\!\right] d\varphi\\
&=&\left(\frac{r}{2\pi}\int_{0}^{2\pi} f(r\cos\varphi)
\sin^{2}\varphi d\varphi\right) \left[\!\!
\begin{array}{r}
\cos\theta\\
\sin\theta
\end{array}\!\!\right]\\
&&\qquad+\left(\frac{r}{2\pi}\int_{0}^{2\pi} f(r\cos\varphi)
\sin\varphi\cos\varphi d\varphi\right) \left[\!\!
\begin{array}{r}
\sin\theta\\
-\cos\theta
\end{array}\!\!\right]
\end{eqnarray*}
where the second term is zero since $f(r\cos\varphi) \sin2\varphi$
is an odd function and we can write
\begin{eqnarray*}
\int_{0}^{2\pi} f(r\cos\varphi) \sin\varphi\cos\varphi d\varphi
&=&\frac{1}{2}\int_{-\pi}^{\pi} f(r\cos\varphi) \sin2\varphi
d\varphi\,.
\end{eqnarray*}
Therefore
\begin{eqnarray}\label{eqn:fbar}
\bar f(v) &=&\left(\frac{r}{2\pi}\int_{0}^{2\pi} f(r\cos\varphi)
\sin^{2}\varphi d\varphi\right) \left[\!\!
\begin{array}{r}
\cos\theta\\
\sin\theta
\end{array}\!\!\right]\nonumber\\
&=& \left(\frac{1}{2\pi}\int_{-\pi}^{\pi} f(r\cos\varphi)
r\sin^{2}\varphi d\varphi\right)\frac{v}{|v|}\,.
\end{eqnarray}
Since $f(r\cos\varphi) r\sin^{2}\varphi$ and $f(r\sin\varphi)
r\cos^{2}\varphi$ are even functions we can write
\begin{eqnarray*}
\frac{1}{2\pi}\int_{-\pi}^{\pi} f(r\cos\varphi) r\sin^{2}\varphi
d\varphi &=&\frac{1}{\pi}\int_{0}^{\pi} f(r\cos\varphi)
r\sin^{2}\varphi d\varphi\\
&=&\frac{1}{\pi}\int_{-\pi/2}^{\pi/2} f(r\sin\varphi)
r\cos^{2}\varphi d\varphi\\
&=&\frac{2}{\pi}\int_{0}^{\pi/2} f(r\sin\varphi) r\cos^{2}\varphi
d\varphi\,.
\end{eqnarray*}
Then by change of variables $\sigma:=r\sin\varphi$ we obtain
\begin{eqnarray}\label{eqn:fbar2}
\frac{1}{2\pi}\int_{-\pi}^{\pi} f(r\cos\varphi) r\sin^{2}\varphi\,
d\varphi = \frac{2}{\pi}\int_{0}^{r} f(\sigma)
\sqrt{1-\frac{\sigma^{2}}{r^{2}}} d\sigma\,.
\end{eqnarray}
Combining \eqref{eqn:fbar} and \eqref{eqn:fbar2} gives the result.
\end{proof}

\begin{claim}\label{clm:kappa}
Defined in \eqref{eqn:kappa}, function $\kappa$ satisfies {\em
(L2)}.
\end{claim}

\begin{proof}
Recall that $F(s)=\int_{0}^{s}f$ satisfies (L2). Therefore there
exists $s_{0}>0$ such that $F(s_{0})=0$, $F(s)<0$ for
$s\in(0,\,s_{0})$, $F$ is positive, nondecreasing on
$(s_{0},\,\infty)$, and $F(s)\to\infty$ as $s\to\infty$. Given
$s>0$, let $\delta_{s}:[0,\,s]\to[0,\,1]$ be
$\delta_{s}(\sigma):=\sqrt{1-(\sigma/s)^2}$. Then we can write
$\kappa(s)=\frac{2}{\pi}\int_{0}^{s}f(\sigma)\delta_{s}(\sigma)d\sigma$.
We observe the following properties.
\begin{itemize}
\item[\bf (D1)] Given $s>0$, function $\delta_{s}(\cdot)$ is
strictly decreasing. \item[\bf (D2)] Given $s>t>0$, map
$\sigma\mapsto\delta_{s}(\sigma)-\delta_{t}(\sigma)$ is strictly
increasing on $[0,\,t]$. \item[\bf (D3)] For each $t>0$ there
exists $s>t$ such that $\delta_{s}(\sigma)>\frac{1}{2}$ for
$\sigma\in(0,\,t)$.
\end{itemize}
It follows from integration by parts that
\begin{eqnarray}\label{eqn:un}
\int_{0}^{s}f(\sigma)\delta_{s}(\sigma)d\sigma<0\qquad\forall
s\in(0,\,s_{0})\,.
\end{eqnarray}
Let us convince ourselves that \eqref{eqn:un} is true. Let
$\delta_{s}^{\prime}(\sigma):=d\delta_{s}(\sigma)/d\sigma$. Given
$s\in(0,\,s_{0})$ we can write
\begin{eqnarray*}
\int_{0}^{s}f(\sigma)\delta_{s}(\sigma)d\sigma &=&
F(\sigma)\delta_{s}(\sigma)\Big|_{0}^{s}
-\int_{0}^{s}F(\sigma)\delta_{s}^{\prime}(\sigma)d\sigma\\
&=& -\int_{0}^{s}F(\sigma)\delta_{s}^{\prime}(\sigma)d\sigma\\
&<&0
\end{eqnarray*}
since for $\sigma\in(0,\,s)$ we have
$\delta_{s}^{\prime}(\sigma)<0$ by (D1) and $F(\sigma)<0$ by (L2).

Next we show
\begin{eqnarray}\label{eqn:piece1}
\int_{0}^{s_{0}}f(\sigma)\delta_{s}(\sigma)d\sigma
>\int_{0}^{s_{0}}f(\sigma)\delta_{t}(\sigma)d\sigma\quad\mbox{for}\quad
s>t>s_{0}\,.
\end{eqnarray}
Given $s>t>s_{0}$, let
$\Delta(\sigma):=\delta_{s}(\sigma)-\delta_{t}(\sigma)$ and
$\Delta^{\prime}(\sigma):=d\Delta(\sigma)/d\sigma$. We can write
\begin{eqnarray*}
\int_{0}^{s_{0}}f(\sigma)\delta_{s}(\sigma)d\sigma &=&
\int_{0}^{s_{0}}f(\sigma)\Delta(\sigma)d\sigma+\int_{0}^{s_{0}}f(\sigma)\delta_{t}(\sigma)d\sigma\\
&=& F(\sigma)\Delta(\sigma)\Big|_{0}^{s_{0}}
-\int_{0}^{s_{0}}F(\sigma)\Delta^{\prime}(\sigma)d\sigma+\int_{0}^{s_{0}}f(\sigma)\delta_{t}(\sigma)d\sigma\\
&=& -\int_{0}^{s_{0}}F(\sigma)\Delta^{\prime}(\sigma)d\sigma+\int_{0}^{s_{0}}f(\sigma)\delta_{t}(\sigma)d\sigma\\
&>& \int_{0}^{s_{0}}f(\sigma)\delta_{t}(\sigma)d\sigma
\end{eqnarray*}
since for $\sigma\in(0,\,s_{0})$ we have
$\Delta^{\prime}(\sigma)>0$ by (D2) and $F(\sigma)<0$ by (L2).

Since $F$ is positive, nondecreasing on $(s_{0},\,\infty)$, we
have $f(z)>0$ for $z>s_{0}$. Therefore (D2) implies
\begin{eqnarray}\label{eqn:piece2}
\int_{s_0}^{s}f(\sigma)\delta_{s}(\sigma)d\sigma
>\int_{s_0}^{t}f(\sigma)\delta_{t}(\sigma)d\sigma\quad\mbox{for}\quad
s>t>s_{0}\,.
\end{eqnarray}
Combining \eqref{eqn:piece1} and \eqref{eqn:piece2} we obtain
\begin{eqnarray}\label{eqn:deux}
\int_{0}^{s}f(\sigma)\delta_{s}(\sigma)d\sigma
>\int_{0}^{t}f(\sigma)\delta_{t}(\sigma)d\sigma
\quad\mbox{for}\quad s>t>s_0\,.
\end{eqnarray}
 Moreover, since $F(s)\to\infty$ as
$s\to\infty$, (D3) readily yields
\begin{eqnarray}\label{eqn:trois}
\lim_{s\to\infty}\int_{0}^{s}f(\sigma)\delta_{s}(\sigma)d\sigma
=\infty\,.
\end{eqnarray}
Combining \eqref{eqn:un}, \eqref{eqn:deux}, and \eqref{eqn:trois}
gives the result.
\end{proof}

\begin{lemma}\label{lem:keyg}
We have $\displaystyle
\bar\gamma_{ij}(v)=\rho_{ij}(|v|)\frac{v}{|v|}$ where
$\rho_{ij}:\Real_{\geq 0}\to\Real_{\geq 0}$ is
\begin{eqnarray}\label{eqn:keyg}
\rho_{ij}(s):=\frac{1}{2\pi}\int_{0}^{2\pi}
\gamma_{ij}(s\sin\varphi)\sin\varphi d\varphi\,.
\end{eqnarray}
\end{lemma}
\begin{proof}
Given $v\in\Real^{2}$, let $r=|v|$ and $\theta\in[0,\,2\pi)$ be
such that $r[-\cos\theta\ \sin\theta]^{T}=v$. Then, by using
standard trigonometric identities,
\begin{eqnarray}\label{eqn:from}
\bar\gamma_{ij}(v) &=&\frac{1}{2\pi}\int_{0}^{2\pi}
\left[\!\!\begin{array}{r} -\sin\varphi \\ \cos\varphi
\end{array}\!\!\right]\gamma_{ij}
(r(\sin\varphi\cos\theta+\cos\varphi\sin\theta))d\varphi\nonumber\\
&=&\frac{1}{2\pi}\int_{0}^{2\pi}
\left[\!\!\begin{array}{r} -\sin\varphi \\ \cos\varphi
\end{array}\!\!\right]\gamma_{ij}
(r\sin(\varphi+\theta))d\varphi\nonumber\\
&=&\frac{1}{2\pi}\int_{0}^{2\pi}
\left[\!\!\begin{array}{r} -\sin(\varphi-\theta) \\ \cos(\varphi-\theta)
\end{array}\!\!\right]\gamma_{ij}
(r\sin\varphi)d\varphi\nonumber\\
&=&\frac{1}{2\pi}\int_{0}^{2\pi}
\left[\!\!\begin{array}{r} -\sin\varphi\cos\theta+\cos\varphi\sin\theta \\
\cos\varphi\cos\theta+\sin\varphi\sin\theta
\end{array}\!\!\right]\gamma_{ij}
(r\sin\varphi)d\varphi\nonumber\\
&=& \left( \frac{1}{2\pi}\int_{0}^{2\pi} \gamma_{ij}
(r\sin\varphi)\sin\varphi d\varphi\right)
\left[\!\!\begin{array}{r} -\cos\theta \\
\sin\theta
\end{array}\!\!\right]\nonumber\\
&&\qquad +\left( \frac{1}{2\pi}\int_{0}^{2\pi} \gamma_{ij}
(r\sin\varphi)\cos\varphi d\varphi\right)
\left[\!\!\begin{array}{c} \sin\theta \\
\cos\theta
\end{array}\!\!\right]
\end{eqnarray}
We focus on the second term in \eqref{eqn:from}. Observe
that
\begin{eqnarray*}
\int_{0}^{\pi}\gamma_{ij}(r\sin\varphi)\cos\varphi d\varphi
&=&\int_{-\pi/2}^{\pi/2}
\gamma_{ij}\left(r\sin\left(\varphi+\frac{\pi}{2}\right)\right)
\cos\left(\varphi+\frac{\pi}{2}\right)d\varphi\\
&=&0
\end{eqnarray*}
since the integrand is an odd function on the interval of integration.
Likewise,
\begin{eqnarray*}
\int_{\pi}^{2\pi}\gamma_{ij}(r\sin\varphi)\cos\varphi d\varphi
&=&\int_{-\pi/2}^{\pi/2}
\gamma_{ij}\left(r\sin\left(\varphi+\frac{3\pi}{2}\right)\right)
\cos\left(\varphi+\frac{3\pi}{2}\right)d\varphi\\
&=&0\,.
\end{eqnarray*}
Therefore
\begin{eqnarray}\label{eqn:from2}
\lefteqn{\int_{0}^{2\pi}\gamma_{ij}(r\sin\varphi)\cos\varphi d\varphi}\nonumber\\
&&=\int_{0}^{\pi}\gamma_{ij}(r\sin\varphi)\cos\varphi d\varphi
+\int_{\pi}^{2\pi}\gamma_{ij}(r\sin\varphi)\cos\varphi d\varphi\nonumber\\
&&=0\,.
\end{eqnarray}
Combining \eqref{eqn:from} and \eqref{eqn:from2} we obtain
\begin{eqnarray*}
\bar\gamma_{ij}(v) &=& \left( \frac{1}{2\pi}\int_{0}^{2\pi}
\gamma_{ij} (r\sin\varphi)\sin\varphi d\varphi\right)
\left[\!\!\begin{array}{r} -\cos\theta \\
\sin\theta
\end{array}\!\!\right]\\
&=&\rho_{ij}(|v|)\frac{v}{|v|}\,.
\end{eqnarray*}
Hence the result.
\end{proof}

\begin{claim}\label{clm:rhoij}
Defined in \eqref{eqn:keyg}, function $\rho_{ij}(s)\equiv 0$ iff
$\gamma_{ij}=0$. Otherwise, there exists $\alpha\in\K$ such that
$\rho_{ij}(s)\geq\alpha(s)$ for all $s$.
\end{claim}

\begin{proof}
That $\rho_{ij}(s)\equiv 0$ if $\gamma_{ij}=0$ is evident from the
definition. Suppose $\gamma_{ij}\neq 0$. Then by (G2) there exists
$\alpha_{1}\in\K$ such that $|\gamma_{ij}(s)|\geq
\alpha_{1}(|s|)$. Then we can write
\begin{eqnarray*}
\rho_{ij}(s)
&=&\frac{1}{2\pi}\int_{0}^{2\pi}
\gamma_{ij}(s\sin\varphi)\sin\varphi d\varphi\\
&=&\frac{1}{2\pi}\int_{0}^{2\pi}
|\gamma_{ij}(s\sin\varphi)||\sin\varphi| d\varphi\\
&\geq&\frac{1}{2\pi}\int_{0}^{2\pi}
\alpha_{1}(s|\sin\varphi|)|\sin\varphi|d\varphi=:\alpha_{2}(s)\,.
\end{eqnarray*}
Note that $\alpha_{2}$ is a class-$\K$ function. Hence the result.
\end{proof}
\\

In the remainder of the section we generate two results on average
array~\eqref{eqn:lienosccbar}.  In the first of those results
(Theorem~\ref{thm:B}) we establish that each of the
oscillators\footnote{We admit that the word {\em oscillator} may
have been an unfortunate choice here since there is no oscillation
(in the standard meaning of the word) taking place in the average
array. Averaging gets rid of oscillations.} in
\eqref{eqn:lienosccbar} eventually oscillates with some magnitude
no greater than some constant $\rho$ which only depends on
function $f$ and not on initial conditions. In the second result
(Theorem~\ref{thm:R}) we assert that if the initial conditions are
{\em right} then the oscillators eventually synchronize both in
phase and magnitude, where the magnitude equals $\rho$. Initial
conditions' being right roughly corresponds to the following
condition. If we depict each oscillator's initial phase with a
point on the unit circle then those points should all lie in an
open\footnote{That is, endpoints are not included.} semicircle.
Constant $\rho$ we mentioned above is indeed defined as the unique
positive number satisfying
\begin{eqnarray*}
\kappa(\rho)=0\,.
\end{eqnarray*}
Existence and uniqueness of $\rho$ is guaranteed by
Claim~\ref{clm:kappa}. Note then that the first result is about
the asymptotic behavior of the solutions of
system~\eqref{eqn:ellcbar} with respect to the following set
\begin{eqnarray*}
\B:=\{\eta\in\Real^{2m}:|\eta_{i}|\leq \rho\}\,.
\end{eqnarray*}
Likewise, the second result has to do with the asymptotic behavior
of the solutions of system~\eqref{eqn:ellcbar} with respect to
\begin{eqnarray*}
\R:=\{\eta\in\Real^{2m}:\eta_{i}=\eta_{j},\,|\eta_{i}|=\rho\
\mbox{for all}\ i,\,j\}\,.
\end{eqnarray*}

\begin{theorem}\label{thm:B}
Consider system~\eqref{eqn:ellcbar}. Set $\B$ is globally
asymptotically stable.
\end{theorem}

\begin{proof}
We will establish the result by constructing a Lyapunov function.
Claim~\ref{clm:kappa} implies that there exists $\alpha_{1}\in\K$
such that $\alpha_{1}(s-\rho)\leq\kappa(s)$ for $s\geq \rho$. Then
we can find $\alpha_{2}\in\K$ such that
\begin{eqnarray*}
\alpha_{2}(|\eta|_{\B})\leq\alpha_{1}\left(\max\{0,\,\max_{i}\{|\eta_{i}|-\rho\}\}\right)\max_{i}|\eta_{i}|\,.
\end{eqnarray*}
Let our candidate Lyapunov function $V:\Real^{2m}\to\Real_{\geq
0}$ be
\begin{eqnarray*}
V(\eta):=\frac{1}{2}\max\{0,\,\max_{i}\{|\eta_{i}|^{2}-\rho^{2}\}\}\,.
\end{eqnarray*}
Then there exist $\alpha_{3},\,\alpha_{4}\in\Kinf$ such that
\begin{eqnarray}\label{eqn:lyap1}
\alpha_{3}(|\eta|_{\B})\leq V(\eta)\leq\alpha_{4}(|\eta|_{\B})\,.
\end{eqnarray}
Now, given some $\eta$ with $|\eta|_{\B}>0$, let (nonempty) set of
indices $\I$ be such that
$\frac{1}{2}(|\eta_{i}|^{2}-\rho^{2})=V(\eta)$ iff $i\in\I$.
Observe that for $i\in\I$, point $\eta_{i}$ lies on the boundary
of the smallest disk (centered at the origin) that contains all
points $\eta_{j}$. Therefore $i\in\I$ implies
$\eta_{i}^{T}(\eta_{j}-\eta_{i})\leq 0$ for all $j$. Then by
Lemma~\ref{lem:keyf} and Lemma~\ref{lem:keyg} we can write (almost
everywhere)
\begin{eqnarray}\label{eqn:lyap2}
\langle\nabla V(\eta),\,{\bar g}(\eta)\rangle &=& \max_{i\in\I}\
\eta_{i}^{T}\left(-\bar f(\eta_{i})+\sum_{j\neq i}
\bar\gamma_{ij}(\eta_{j}-\eta_{i})\right)\nonumber\\
&=& \max_{i\in\I}\ \left(-\eta_{i}^{T}\bar f(\eta_{i})+\sum_{j\neq
i}
\eta_{i}^{T}\bar\gamma_{ij}(\eta_{j}-\eta_{i})\right)\nonumber\\
&=& \max_{i\in\I}\ \left(-\kappa(|\eta_{i}|)|\eta_{i}|+\sum_{j\neq
i}
\frac{\rho_{ij}(|\eta_{j}-\eta_{i}|)}{|\eta_{j}-\eta_{i}|}\eta_{i}^{T}(\eta_{j}-\eta_{i})\right)\nonumber\\
&\leq& -\kappa\left(\max_{i}|\eta_{i}|\right)\max_{i}|\eta_{i}|\nonumber\\
&\leq& -\alpha_{1}\left(\max_{i}\{|\eta_{i}|-r\}\right)\max_{i}|\eta_{i}|\nonumber\\
&\leq& -\alpha_{2}(|\eta|_{\B})\,.
\end{eqnarray}
Result follows from \eqref{eqn:lyap1} and \eqref{eqn:lyap2}.
\end{proof}

\begin{remark}
Theorem~\ref{thm:B} does not require interconnection
$\{\gamma_{ij}\}$ to be connected. This is clear with (or even
without) the proof.
\end{remark}

\begin{theorem}\label{thm:R}
Consider system~\eqref{eqn:ellcbar}. Set $\R$ is locally
asymptotically stable. In particular,
$\co\{\eta_{1}(0),\,\ldots,\,\eta_{m}(0)\}\cap\{0\}=\emptyset$
implies $\eta(t)\to\eta^{*}$ as $t\to\infty$ for some
$\eta^{*}\in\R$.
\end{theorem}

To prove the theorem we use an invariance principle (similar to
that of LaSalle's) for which we need to tailor an invariant set
for our system. To this end we introduce some notation. Given
$\eta\in\Real^{2m}$ with
$\co\{\eta_{1},\,\ldots,\,\eta_{m}\}\cap\{0\}=\emptyset$ we let
$\C(\eta)\subset\Real^{2}$ and $\W_{\rho}(\eta)\subset\Real^{2}$
respectively be the smallest cone and smallest $\rho$-wedge
containing set $\{\eta_{1},\,\ldots,\,\eta_{m}\}$. (Recall that
$\rho>0$ satisfies $\kappa(\rho)=0$.) Fig.~5 depicts
$\W_{\rho}(\eta)$. Note that
\begin{eqnarray}\label{eqn:notethat}
\W_{\rho}(\eta)\subset\C(\eta)\,.
\end{eqnarray}
For $\zeta\in\Real^{2m}$, when we write
$\zeta\in\W_{\rho}(\eta)^{m}$ we mean that
$\zeta_{i}\in\W_{\rho}(\eta)$ for all $i=1,\,\ldots,\,m$. Note
that
\begin{eqnarray}\label{eqn:implies}
\zeta\in\W_{\rho}(\eta)^{m} \implies
\W_{\rho}(\zeta)\subset\W_{\rho}(\eta)\,.
\end{eqnarray}

\begin{figure}[h]
\begin{center}
\includegraphics[scale=0.6]{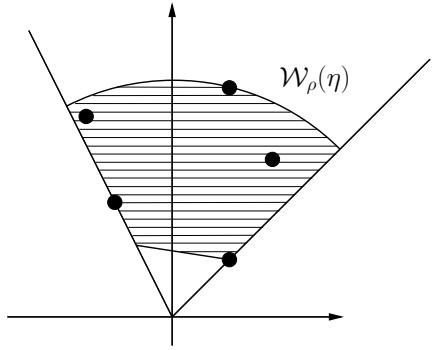}
\caption{An example with $\eta\in\Real^{10}$. Dots represent the
positions of $\eta_{i}$ on plane. The shaded region is
$\W_{\rho}(\eta)$.}
\end{center}
\end{figure}

\noindent {\bf Proof of Theorem~\ref{thm:R}.} We first show local
stability of $\R$. Let $\eta\in\Real^{2m}$ be such that
$\co\{\eta_{1},\,\ldots,\,\eta_{m}\}\cap\{0\}=\emptyset$. Consider
the first term of the righthand side in \eqref{eqn:lienosccbar}.
Lemma~\ref{lem:keyf} and Claim~\ref{clm:kappa} tell us that vector
$-\bar{f}(\eta_{i})$ (tail placed at point $\eta_{i}$) points
toward the origin if $|\eta_{i}|>\rho$ and away from the origin if
$|\eta_{i}|<\rho$. Therefore, for $\eta_{i}$ on the boundary of
$\W_{\rho}(\eta)$, vector $-\bar{f}(\eta_{i})$ never points
outside $\W_{\rho}(\eta)$. Now consider the sum term in
\eqref{eqn:lienosccbar}. Lemma~\ref{lem:keyg} and
Claim~\ref{clm:rhoij} tell us that each summand
$\bar\gamma_{ij}(\eta_{j}-\eta_{i})$, if nonzero, is a vector
pointing from $\eta_{i}$ to $\eta_{j}$. Hence sum $\sum_{j\neq
i}\bar\gamma_{ij}(\eta_{j}-\eta_{i})$ cannot be pointing outside
convex hull $\co\{\eta_{1},\,\ldots,\,\eta_{m}\}$ for $\eta_{i}$
on the boundary of the convex hull. Since
$\co\{\eta_{1},\,\ldots,\,\eta_{m}\}\subset\W_{\rho}(\eta)$ we
deduce therefore that the righthand side of
\eqref{eqn:lienosccbar}, that is, vector $-\bar
f(\eta_{i})+\sum_{j\neq i} \bar\gamma_{ij}(\eta_{j}-\eta_{i})$,
never points outside $\W_{\rho}(\eta)$ for $\eta_{i}$ on the
boundary of $\W_{\rho}(\eta)$. Hence compact set
$\W_{\rho}(\eta)^{m}$ is forward invariant with respect to system
\eqref{eqn:ellcbar}. Observe also that for $\eta^{*}\in\R$ we have
$\W_{\rho}(\eta)^{m}\to\{\eta^{*}\}$ as $\eta\to\eta^{*}$. Hence
set $\R$ is locally stable. Next we show attractivity.

Let $G$ denote the graph of interconnection $\{\gamma_{ij}\}$.
Note that $G$ is connected by assumption. By Lemma~\ref{lem:keyg}
and Claim~\ref{clm:rhoij} we know that $\rho_{ij}$ is continuous
and zero at zero. Moreover, if there is no edge of $G$ from node
$i$ to node $j$ then $\rho_{ij}(s)\equiv 0$. Otherwise there
exists $\alpha\in\K$ such that $\rho_{ij}(s)\geq\alpha(s)$ for all
$s$.

Consider system~\eqref{eqn:ellcbar}. Let
$\co\{\eta_{1}(0),\,\ldots,\,\eta_{m}(0)\}\cap\{0\}=\emptyset$.
Due to \eqref{eqn:implies} and forward invariance of
$\W_{\rho}(\eta(t))^{m}$, for $t_{2}\geq t_{1}\geq 0$ we can write
$\W_{\rho}(\eta(t_{2}))\subset\W_{\rho}(\eta(t_{1}))$. By
\eqref{eqn:notethat} map $t\mapsto\angle\C(\eta(t))$ is hence
nonincreasing. It is also bounded from below by definition.
Therefore there exists $\theta\in[0,\,\pi)$ such that
$\lim_{t\to\infty}\angle\C(\eta(t))=\theta$. We now claim that
$\theta=0$.

Suppose not. Then, by continuity, there must exist a solution
$\zeta(\cdot)$ to system~\eqref{eqn:ellcbar} with
$\zeta(0)\in\W_{\rho}(\eta(0))^{m}$ such that
$\angle\C(\zeta(t))=\theta$ for all $t\geq 0$. Recall that
$\W_{\rho}(\zeta(0))$ is forward invariant. By
\eqref{eqn:notethat} therefore
\begin{eqnarray}\label{eqn:therefore}
\C(\zeta(t))=\C(\zeta(0))\qquad \forall t\geq 0\,.
\end{eqnarray}
Let $\setH_{1}$ and $\setH_{2}$ be two half lines convex hull of
which equals $\C(\zeta(0))$. Since $\theta>0$, we have
$\setH_{1}\neq\setH_{2}$. Then we observe that
$\{\zeta_{1}(0),\,\ldots,\,\zeta_{m}(0)\}\cap\setH_{i}\neq\emptyset$
for $i=1,\,2$. Let $n_{1},\,\ldots,\,n_{m}$ denote the nodes of
graph $G$. Since $G$ is connected we can find a node $n_{l}$ to
which there exists a path from every other node. Now, without loss
of generality, assume that $\zeta_{l}(0)\notin\setH_{1}$. Let
$\I_{1}(t)$ denote the set of indices at time $t\geq 0$ such that
$\zeta_{i}(t)\in\setH_{1}$ iff $i\in\I_{1}(t)$. Clearly,
$l\notin\I_{1}(0)$ and $1\leq \#\I_{1}(0)\leq m-1$. Hence
connectedness of $G$ implies that there exist $i_{0}\in\I_{1}(0)$
and $j_{0}\in\{1,\,\ldots,\,m\}\setminus\I_{1}(0)$ such that
$(n_{i_{0}},\,n_{j_{0}})$ is an edge of $G$. Existence of edge
$(n_{i_{0}},\,n_{j_{0}})$ implies that $\gamma_{i_{0}j_{0}}\neq
0$. Now, observe that $\zeta_{j_{0}}(0)-\zeta_{i_{0}}(0)\neq 0$
since $\zeta_{i_{0}}(0)\in\setH_{1}$ and
$\zeta_{j_{0}}(0)\notin\setH_{1}$. By Lemma~\ref{lem:keyg},
$\bar\gamma_{i_{0}j_{0}}(\zeta_{j_{0}}(0)-\zeta_{i_{0}}(0))$ is a
nonzero vector that is not parallel to $\setH_{1}$. By
\eqref{eqn:lienosccbar} we can write
\begin{eqnarray*}
{\dot\zeta}_{i_{0}}(0)=v+\bar\gamma_{i_{0}j_{0}}(\zeta_{j_{0}}(0)-\zeta_{i_{0}}(0))
\end{eqnarray*}
where
\begin{eqnarray*}
v=-\bar f(\zeta_{i_{0}}(0))+\sum_{j\neq i_{0},j_{0}}
\bar\gamma_{i_{0}j}(\zeta_{j}(0)-\zeta_{i_{0}}(0))\,.
\end{eqnarray*}
By earlier arguments (see the first paragraph of the proof) we
know that vector $v$ (tail placed at point $\zeta_{i}(0)$) cannot
point outside $\C(\zeta(0))$. Therefore vector
${\dot\zeta}_{i_{0}}(0)$ is nonzero and not parallel to
$\setH_{1}$. This lets us be able to find $t_{1}>0$ such that
$\zeta_{i_{0}}(t_{1})\notin\setH_{1}$ and $\#\I_{1}(t_{1})\leq
\#\I_{1}(0)-1$. We can continue the procedure and find a sequence
of instants $t_{k}>0$ such that $\#\I_{1}(t_{k})\leq
\#\I_{1}(0)-k$. Since $\#\I_{1}(0)$ is finite, there should exist
$t^{*}>0$ such that $\#\I_{1}(t^{*})=0$, that is,
$\I_{1}(t^{*})=\emptyset$. This observation translates to
$\setH_{1}\cap\{\zeta_{1}(t^{*}),\,\ldots,\,\zeta_{m}(t^{*})\}=\emptyset$,
which yields $\C(\zeta(t^{*}))\neq\C(\zeta(0))$ which contradicts
with \eqref{eqn:therefore}. Hence our claim is valid and
$\lim_{t\to\infty}\angle\C(\eta(t))=0$.

That $\angle\C(\eta(t))\to 0$ means that $\C(\eta(t))\to\setH$ for
some half line $\setH$. We also know that
$\eta(t)\in\W_{\rho}(\eta(0))^{m}$ for all $t\geq 0$ and, by
Theorem~\ref{thm:B}, $\eta(t)\to\B$. Therefore solutions
$\eta_{i}(t)$, for $i=1,\,\ldots,\,m$, converge to the following
set
\begin{eqnarray*}
{\mathcal
S}:=\setH\cap\W_{\rho}(\eta(0))\cap\{v\in\Real^{2}:|v|\leq
\rho\}\,.
\end{eqnarray*}
Note that ${\mathcal S}$ is a line segment that satisfies
${\mathcal S}=\{\lambda u:\delta\leq \lambda\leq \rho\}$ for some
unit vector $u\in\Real^{2}$ and $\delta>0$. Solution $\eta(\cdot)$
will converge to the largest invariant set in ${\mathcal S}^{m}$.
(Note that ${\mathcal S}^{m}$ itself is forward invariant.) Take
any solution $\zeta(\cdot)$ to system~\eqref{eqn:ellcbar} with
$\zeta(0)\in{\mathcal S}^{m}$. At any given time $t\geq 0$ let $i$
be such that point $\zeta_{i}(t)$ is farthest from the point $\rho
u$. Then by Lemma~\ref{lem:keyf}, Claim~\ref{clm:kappa},
Lemma~\ref{lem:keyg}, and Claim~\ref{clm:rhoij} we have
\begin{eqnarray*}
\dot\zeta_{i}(t)=\left\{
\begin{array}{ccc}
0&\mbox{if}&\zeta_{i}(t)=\rho u\\
\lambda u&\mbox{if}&\zeta_{i}(t)\neq \rho u
\end{array}
\right.
\end{eqnarray*}
for some $\lambda\geq -\kappa(|\zeta_{i}(t)|)>0$. Therefore
$\max_{i}|\zeta_{i}(t)- \rho u|\to 0$ as $t\to\infty$. We then
deduce that $\{\rho u\}^{m}$ is the largest invariant set in
${\mathcal S}^{m}$. Hence the result. \hfill\null$\blacksquare$

\begin{remark}
In Theorem~\ref{thm:R} we require condition
$\co\{\eta_{1}(0),\,\ldots,\,\eta_{m}(0)\}\cap\{0\}=\emptyset$ for
synchronization. If we express each $\eta_{i}$ in polar
coordinates $(r_{i},\,\theta_{i})$, where $r_{i}\geq 0$ and
$\theta_{i}\in[0,\,2\pi)$, then the required condition is
equivalent to that $r_{i}(0)\neq 0$ and
$\theta_{i}(0)\in(\theta^{*},\,\theta^{*}+\pi)$ for all $i$ and
some $\theta^{*}$. In other words, if we denote each phase
$\theta_{i}(0)$ by a point on the unit circle, then the condition
is equivalent to that they all lie in an open semicircle. Is this
condition really essential for synchronization? Theoretically
speaking, yes. That is, one can easily find initial conditions
that make an equilibrium but are not on the synchronization
manifold. However those equilibria may be practically
disregardable if they are unstable, since slightest disturbance
would rescue the system from being stuck at those points. Now we
ask the second question. Does there exist a stable equilibrium
outside synchronization manifold? This we find difficult to
answer. We would nevertheless like to report that our attempts to
hunt one via (limited) simulations on coupled van der Pol
oscillators have failed.
\end{remark}

Let us recapitulate what has hitherto been done. We have begun by
an array of coupled Lienard oscillators
$\dot\xi=\ell(\xi,\,\omega)$ with frequency of oscillations
$\omega$. By looking at it from a rotating (at frequency $\omega$)
reference frame we have obtained new array $\dot
x=\ellc(x,\,\omega t)$ whose righthand side is periodic in time.
Exploiting this periodicity, we have then computed average
dynamics $\dot \eta=\ellcbar(\eta)$. Finally we have shown for the
average array that if oscillators' initial phases all reside in an
open semicircle then they synchronize both in phase and magnitude,
where the magnitude should equal $\rho$.

In the next section, based on our findings on the synchronization
behavior of average array, we show that oscillators of array $\dot
x=\ellc(x,\,\omega t)$ and therefore of array
$\dot\xi=\ell(\xi,\,\omega)$ get arbitrarily close to
synchronization for $\omega$ large enough. We use tools from
averaging theory to establish the results.

\section{Synchronization of Lienard oscillators}\label{sec:sync}

The following definition is borrowed with slight modification from
\cite{teel99}.

\begin{definition}
Consider system ${\dot x}=g(x,\,\omega,\, t)$. Closed set
${\mathcal S}$ is said to be {\em semiglobally practically
asymptotically stable (with respect to $\omega$)} if for each pair
$(\Delta,\,\delta)$ of positive numbers, there exists
$\omega^{*}>0$ such that for all $\omega\geq\omega^{*}$ the
following hold.
\begin{itemize}
\item[(a)] For each $r>\delta$ there exists $\varepsilon>0$ such
that
\begin{equation*}
|x(t_0)|_{\mathcal S}\leq \varepsilon \implies |x(t)|_{\mathcal
S}\leq r \quad \forall t\geq t_0\,.
\end{equation*}
\item[(b)] For each $\varepsilon<\Delta$ there exists $r>0$ such
that
\begin{equation*}
|x(t_0)|_{\mathcal S}\leq \varepsilon \implies  |x(t)|_{\mathcal
S}\leq r \quad \forall t\geq t_0\,.
\end{equation*}
\item[(c)] For each $r<\Delta$ and $\varepsilon>\delta$ there
exists $T>0$ such that
\begin{equation*}
|x(t_0)|_{\mathcal S}\leq r \implies |x(t)|_{\mathcal
S}\leq\varepsilon \quad \forall t\geq t_{0}+T\,.
\end{equation*}
\end{itemize}
\end{definition}

Result \cite[Thm.~2]{teel99} tells that the origin of a
time-varying system is semiglobally practically asymptotically
stable provided that the origin of the average system is globally
asymptotically stable. The analysis therein, with almost no extra
effort, can be extended to cover the case where the attractor is
not a singleton but only compact. Note that set $\B$ is compact.
Theorem~\ref{thm:B} therefore yields the following result.

\begin{theorem}\label{thm:mainB}
Consider system~\eqref{eqn:ellc}. Set $\B$ is semiglobally
practically asymptotically stable.
\end{theorem}

The following result is an adaptation of a general theorem on
averaging \cite[Thm.~2.8.1]{sanders07}.

\begin{lemma}\label{lem:averaging}
Let map $g:\Real^{n}\times\Real\to\Real^{n}$ be locally Lipschitz
and satisfy $g(x,\,\varphi+2\pi)=g(x,\,\varphi)$ for all
$x\in\Real^{n}$ and $\varphi\in\Real$. Define average function
${\bar g}:\Real^{n}\to\Real^{n}$ as
\begin{eqnarray*}
{\bar
g}(x):=\frac{1}{2\pi}\int_{0}^{2\pi}g(x,\,\varphi)d\varphi\,.
\end{eqnarray*}
Let $x(\cdot)$ and $\eta(\cdot)$ denote, respectively, the
solutions of systems $\dot{x}=g(x,\,\omega t)$ and
$\dot{\eta}=\bar{g}(\eta)$, where $\omega>0$. Then for each
compact set $\D\subset\Real^{n}$ and pair of positive real numbers
$(T,\,\varepsilon)$ there exists $\omega^{*}>0$ such that if
\begin{itemize}
\item $\omega\geq\omega^{*}$, \item $\eta(0)=x(0)$, and \item
$x(t)\in\D$ for all $t\in[0,\,T]$
\end{itemize}
hold then $|\eta(t)-x(t)|\leq\varepsilon$ for all $t\in[0,\,T]$.
\end{lemma}

Lemma~\ref{lem:averaging} lets us establish the following result.

\begin{theorem}\label{thm:mainR}
Consider system~\eqref{eqn:ellc}. For each pair $(\D,\,\delta)$,
where $\D\subset\Real^{2}$ is a compact convex set that does not
include the origin and $\delta>0$, there exists $\omega^{*}>0$
such that for all $\omega\geq\omega^{*}$ the following hold.
\begin{itemize}
\item[(a)] There exists $\varepsilon>0$ such that
\begin{equation*}
|x(0)|_{\R}\leq \varepsilon \implies |x(t)|_{\R}\leq \delta \quad
\forall t\geq 0\,.
\end{equation*}
\item[(b)] There exists $r>0$ such that
\begin{equation*}
x(0)\in\D^{m} \implies  |x(t)|_{\R}\leq r \quad \forall t\geq 0\,.
\end{equation*}
\item[(c)] There exists $T>0$ such that
\begin{equation*}
x(0)\in\D^{m} \implies |x(t)|_{\R}\leq\delta \quad \forall t\geq
T\,.
\end{equation*}
\end{itemize}
\end{theorem}

\begin{proof}
Let us be given $(\D,\,\delta)$. Let $x(\cdot)$ and $\eta(\cdot)$
respectively denote the solutions of system~\eqref{eqn:ellc} and
system~\eqref{eqn:ellcbar}. We first work on part (a). Choose some
$\delta_{1}>0$ such that $|x|_{\R}\leq\delta_{1}$ implies
$\co\{x_{1},\,\ldots,\,x_{m}\}\cap\{0\}=\emptyset$. Then set
$\delta_{2}:=\frac{1}{2}\min\{\delta_{1},\,\delta\}$. Choose
$\varepsilon\in(0,\,\delta_{2})$ such that
\begin{eqnarray*}
|\eta(0)|_{\R}\leq\varepsilon\implies|\eta(t)|_{\R}\leq\delta_{2}\quad\forall
t\geq 0\,.
\end{eqnarray*}
Such $\varepsilon$ exists since $\R$ is locally asymptotically
stable for system~\eqref{eqn:ellcbar} by Theorem~\ref{thm:R}.
Moreover, $|\eta(0)|_{\R}\leq\varepsilon$ implies
$\co\{\eta_{1}(0),\,\ldots,\,\eta_{m}(0)\}\cap\{0\}=\emptyset$.
Therefore we can find $T_{1}>0$ such that
\begin{eqnarray*}
|\eta(0)|_{\R}\leq\varepsilon\implies|\eta(t)|_{\R}\leq\frac{\varepsilon}{2}\quad\forall
t\geq T_{1}\,.
\end{eqnarray*}
Let $\varepsilon_{1}:=\min\{\frac{\varepsilon}{2},\,\delta_{2}\}$.
By Lemma~\ref{lem:averaging} there exists $\omega_{1}>0$ such that
$|x(t)-\eta(t)|\leq\varepsilon_{1}$ for all $t\in[0,\,T_{1}]$
provided that $\omega\geq\omega_{1}$, $|x(t)|_{\R}\leq\delta$ for
all $t\in[0,\,T_{1}]$, and $x(0)=\eta(0)$. Let now
$\omega\geq\omega_{1}$ and $|x(0)|_{\R}\leq\varepsilon$. Set
$\eta(0)=x(0)$. Then for all $t\in[0,\,T_{1}]$ we can write
\begin{eqnarray*}
|x(t)|_{\R}&\leq&|\eta(t)|_{\R}+|x(t)-\eta(t)|\\
&\leq&\delta_{2}+\varepsilon_{1}\\
&\leq&\delta\,.
\end{eqnarray*}
Also
\begin{eqnarray*}
|x(T_{1})|_{\R}&\leq&|\eta(T_{1})|_{\R}+|x(T_{1})-\eta(T_{1})|\\
&\leq&\frac{\varepsilon}{2}+\varepsilon_{1}\\
&\leq&\varepsilon\,.
\end{eqnarray*}
We can repeat this procedure for the following intervals
$[kT_{1},\,(k+1)T_{1}]$, $k=1,\,2,\,\ldots$ Therefore
$|x(0)|_{\R}\leq\varepsilon$ implies $|x(t)|_{\R}\leq\delta$ as
long as $\omega\geq \omega_{1}$.

Part (b). By Theorem~\ref{thm:mainB} we can find $\omega_{2}>0$
and $r_{1}>0$ such that $\omega\geq\omega_{2}$ and $x(0)\in\D^{m}$
imply $|x(t)|_{\B}\leq r_{1}$ for all $t\geq 0$. Then we can find
$r>0$ such that $|x|_{\B}\leq r_{1}$ implies $|x|_{\R}\leq r$.
Therefore $\omega\geq\omega_{2}$ and $x(0)\in\D^{m}$ imply
$|x(t)|_{\R}\leq r$ for all $t\geq 0$.

Part (c). First recall, from the proof of Theorem~\ref{thm:mainR},
that any $\rho$-wedge is forward invariant with respect to
system~\eqref{eqn:ellcbar}. Choose a $\rho$-wedge $\W_{\rho}$ such
that $\D\subset\W_{\rho}$. Note that $\W_{\rho}$ is compact,
convex, and does not include the origin. By
Theorem~\ref{thm:mainR} there exists $T>0$ such that
\begin{eqnarray*}
\eta(0)\in\W_{\rho}^{m} \implies |\eta(T)|_{\R}\leq\varepsilon_{1}
\end{eqnarray*}
where $\varepsilon_{1}$ is as found in part (a). Then define
$\setS\subset\Real^{2m}$ as
\begin{eqnarray*}
\setS:=\{x+\varepsilon_{1}v:x\in\W_{\rho}^{m},\,|v|\leq 1\}\,.
\end{eqnarray*}
By Lemma~\ref{lem:averaging} there exists $\omega_{3}>0$ such that
$|x(t)-\eta(t)|\leq\varepsilon_{1}$ for all $t\in[0,\,T]$ provided
that $\omega\geq\omega_{3}$, $x(t)\in\setS$ for all $t\in[0,\,T]$,
and $x(0)=\eta(0)$. Let now $\omega\geq\omega_{3}$ and
$x(0)\in\D^{m}$. Set $\eta(0)=x(0)$. Then for all $t\in[0,\,T]$ we
have $|x(t)-\eta(t)|\leq\varepsilon_{1}$. Therefore
\begin{eqnarray*}
|x(T)|_{\R}&\leq&|\eta(T)|_{\R}+|x(T)-\eta(T)|\\
&\leq&\varepsilon_{1}+\varepsilon_{1}\\
&\leq&\varepsilon\,.
\end{eqnarray*}
From part (a) we know that $|x(T)|_{\R}\leq\varepsilon$ implies
$|x(t)|_{\R}\leq\delta$ for all $t\geq T$, which was to be shown.

We complete the proof by setting
$\omega^{*}:=\max\{\omega_{1},\,\omega_{2},\,\omega_{3}\}$.
\end{proof}
\\

We next present the below result for the array of coupled Lienard
oscillators. This result directly follow from the observation
mentioned in Remark~\ref{rem:arkabak} and the fact that sets $\B$
and $\R$ are invariant under rotations in $\Real^{2}$.

\begin{theorem}\label{thm:lienBnR}
Theorem~\ref{thm:mainB} and Theorem~\ref{thm:mainR} hold true for
system~\eqref{eqn:ell}.
\end{theorem}

Theorem~\ref{thm:mainB} roughly says that solutions $x_{i}(\cdot)$
of oscillators of array~\eqref{eqn:lienoscc} can be made converge
to an arbitrarily small neighborhood of the disk
$\{v\in\Real^{2}:|v|\leq\rho\}$ starting from arbitrarily large
initial conditions by choosing oscillation frequency $\omega$
arbitrarily large. Theorem~\ref{thm:mainR} says that by choosing
$\omega$ arbitrarily large, solutions $x_{i}(\cdot)$ can be made
eventually become arbitrarily close to each other and to the ring
$\{v\in\Real^{2}:|v|=\rho\}$, starting from within an arbitrary
convex compact set that does not contain the origin. Finally
Theorem~\ref{thm:lienBnR} says that these two results should hold
also for array of Lienard oscillators~\eqref{eqn:lienosc}.

Theorem~\ref{thm:lienBnR} marks the end of our answer to the
question that we asked for coupled Lienard oscillators in
Section~\ref{sec:ps}. It is hard not to realize that the basic
idea forming the skeleton of this answer serves as a solution
approach also for the problem of understanding the synchronization
behavior of coupled harmonic oscillators. Therefore in the next
section we analyze the relation between the frequency of
oscillations and the synchronization in an array of harmonic
oscillators. The results of next section will be similar to the
results of previous sections, however there will be two main
differences. The first difference is that, unlike Lienard
oscillators, where the magnitude of oscillations at
synchronization is fixed, i.e., independent of initial conditions,
with harmonic oscillators synchronization can occur at any
magnitude depending on the initial conditions. The second
difference is that the initial phases of harmonic oscillators do
not play any role in determining whether synchronization will take
place or not provided that the frequency of oscillations is large
enough. Recall that this was not the case with Lienard
oscillators.

\section{Synchronization of harmonic oscillators}\label{sec:harmonic}

Consider the following array of coupled harmonic oscillators
\begin{subeqnarray}\label{eqn:harmosc}
\dot{q}_{i}&=&\omega p_{i}\\
\dot{p}_{i}&=&-\omega q_{i}+\sum_{j\neq
i}\gamma_{ij}(p_{j}-p_{i})\,, \qquad i=1,\,\ldots,\,m
\end{subeqnarray}
where, as before, $\omega>0$ is the frequency of oscillations and
$\{\gamma_{ij}\}$ is a connected interconnection. Functions
$\gamma_{ij}$ are assumed to be locally Lipschitz. We let
$\xi_{i}=[q_{i}\ p_{i}]^{T}$ denote the state of $i$th oscillator.
Array~\eqref{eqn:harmosc} defines the below system in $\Real^{2m}$
\begin{eqnarray}\label{eqn:h}
\dot\xi=h(\xi,\,\omega)
\end{eqnarray}
where $\xi=[\xi_{1}^{T}\,\ldots\,\xi_{m}^{T}]^{T}$ and what $h$ is
should be clear. Following the same procedure adopted for Lienard
array, applying change of coordinates
$x_{i}(t)=e^{-S(\omega)t}\xi_{i}(t)$ yields
\begin{eqnarray}\label{eqn:harmoscc}
\dot{x}_{i}=\sum_{j\neq i}\left[\!\!\begin{array}{r} -\sin\omega t\\
\cos\omega t
\end{array}\!\!\right]\gamma_{ij}
([-\sin\omega t\ \cos\omega t](x_{j}-x_{i}))
\end{eqnarray}
which defines the below system in $\Real^{2m}$
\begin{eqnarray}\label{eqn:hc}
\dot{x}=\hc(x,\,\omega t)
\end{eqnarray}
for $x=[x_{1}^{T}\,\ldots\,x_{m}^{T}]^{T}$. Then, due to the
periodicity of righthand side of \eqref{eqn:harmoscc} we can talk
about the {\em average} array
\begin{eqnarray}\label{eqn:harmosccbar}
{\dot\eta}_{i}=\sum_{j\neq i} \bar\gamma_{ij}(\eta_{j}-\eta_{i})
\end{eqnarray}
where $\bar\gamma_{ij}$ is as defined in \eqref{eqn:gammaijbar}.
Array~\eqref{eqn:harmosccbar} defines the below system in
$\Real^{2m}$
\begin{eqnarray}\label{eqn:hcbar}
\dot{\eta}=\hcbar(\eta)
\end{eqnarray}
for $\eta=[\eta_{1}^{T}\,\ldots\,\eta_{m}^{T}]^{T}$.

In the remainder of the section we establish via sequence of three
theorems the relation between frequency of oscillations and
synchronization of coupled harmonic oscillators. In the first of
those results (Theorem~\ref{thm:Acbar}) we show that the
oscillators of average array~\eqref{eqn:harmosccbar} globally
synchronize. Then from our first result we deduce (in
Theorem~\ref{thm:Ac}) that solutions $x_{i}(\cdot)$ of
array~\eqref{eqn:harmoscc} should eventually become arbitrarily
close to each other, while initially being arbitrarily far from
each other, provided that frequency of oscillations is arbitrarily
large. Finally we claim (in Theorem~\ref{thm:A}) that what is true
for array~\eqref{eqn:harmoscc} is also true for
array~\eqref{eqn:harmosc} due to the norm-preserving nature of the
transformation being used to transition between two arrays.

Let us now define synchronization manifold $\A\subset\Real^{2m}$
as
\begin{eqnarray*}
\A:=\{\eta\in\Real^{2m}:\eta_{i}=\eta_{j}\ \mbox{for all}\ i,\,j\}
\end{eqnarray*}
to be used in the theorems to follow.

\begin{theorem}\label{thm:Acbar}
Consider system~\eqref{eqn:hcbar}. Synchronization manifold $\A$
is globally asymptotically stable.
\end{theorem}

\begin{proof}
By Lemma~\ref{lem:keyg} we can write
\begin{eqnarray*}
\dot\eta_{i}=\sum_{j\neq i} \rho_{ij}
(|\eta_{j}-\eta_{i}|)\frac{\eta_{j}-\eta_{i}}{|\eta_{j}-\eta_{i}|}=:\hcbar_{i}(\eta)
\end{eqnarray*}
for $i=1,\,\ldots,\,m$. Let $\G$ denote the graph of
interconnection $\{\gamma_{ij}\}$. Graph $\G$ is connected by
assumption. Note that $\rho_{ij}$ is continuous and zero at zero.
Also, by Claim~\ref{clm:rhoij}, if there is no edge of $\G$ from
node $i$ to node $j$ then $\rho_{ij}(s)\equiv 0$. If there is an
edge from node $i$ to node $j$ then $\rho_{ij}(s)>0$ for $s>0$.

Therefore $\hcbar_{i}$ is continuous; and vector
$\hcbar_{i}(\eta)$ always points to the (relative) interior of the
convex hull of the set $\{\eta_{i}\}\cup\{\eta_{j}:\mbox{there is
an edge of $\G$ from node $i$}$ $\mbox{to node $j$}\}$. These two
conditions together with connectedness of $\G$ yield by
\cite[Cor.~3.9]{lin07} that system~\eqref{eqn:ellcbar} has the
{\em globally asymptotic state agreement property}, see
\cite[Def.~3.4]{lin07}. Another property of the system is {\em
invariance with respect to translations}. That is,
$\hcbar(\eta+\zeta)=\hcbar(\eta)$ for $\zeta\in\A$. These
properties let us write the following.
\begin{itemize}
\item[(a)] There exists $\alpha\in\K$ such that
$|\eta(t)|_{\A}\leq\alpha(|\eta(0)|_{\A})$ for all $t\geq 0$.
\item[(b)] For each $r>0$ and $\varepsilon>0$, there exists $T>0$
such that $|\eta(0)|_{\A}\leq r$ implies
$|\eta(t)|_{\A}\leq\varepsilon$ for all $t\geq T$.
\end{itemize}
Finally, (a) and (b) give us the result by \cite[Prop.~1]{teel00}.
\end{proof}

\begin{theorem}\label{thm:Ac}
Consider system~\eqref{eqn:hc}. Synchronization manifold $\A$ is
semiglobally practically asymptotically stable.
\end{theorem}

\begin{proof}
Consider array~\eqref{eqn:harmoscc}. Observe that the righthand
side depends only on the relative distances $x_{j}-x_{i}$. This
allows us to reduce the order of the system. Let us define
\begin{eqnarray*}
y:=\left[\!\!
\begin{array}{c}
x_{2}-x_{1}\\
x_{3}-x_{1}\\
\vdots\\
x_{m}-x_{1}
\end{array}
\!\!\right]
\end{eqnarray*}
Note that $\dot y=h^{\rm r}(y,\,\omega t)$ for some $h^{\rm
r}:\Real^{2m-2} \times\Real_{\geq 0}\to \Real^{2m-2}$. Since
functions $\gamma_{ij}$ are assumed to be locally Lipschitz,
$h^{\rm r}$ is locally Lipschitz in $y$ uniformly in $t$. Also,
$h^{\rm r}$ is periodic in time by \eqref{eqn:harmoscc}. Now
consider array~\eqref{eqn:harmosccbar}. Again the righthand side
depends only on the relative distances $\eta_{j}-\eta_{i}$. Define
\begin{eqnarray*}
z:=\left[\!\!
\begin{array}{c}
\eta_{2}-\eta_{1}\\
\eta_{3}-\eta_{1}\\
\vdots\\
\eta_{m}-\eta_{1}
\end{array}
\!\!\right]
\end{eqnarray*}
Then $\dot z=\bar h^{\rm r}(z)$ where $\bar h^{\rm r}$ is the time
average of $h^{\rm r}$ and locally Lipschitz both due to that
$\bar\gamma_{ij}$ is the time average of $\gamma_{ij}$.

Theorem~\ref{thm:Acbar} implies that the origin of $\dot z=\bar
h^{\rm r}(z)$ is globally asymptotically stable. Then
\cite[Thm.~2]{teel99} tells us that the origin of $\dot y=h^{\rm
r}(y,\,\omega t)$ is semiglobally practically asymptotically
stable. All there is left to complete the proof is to realize that
semiglobal practical asymptotic stability of the origin of $\dot
y=h^{\rm r}(y,\,\omega t)$ is equivalent to semiglobal practical
asymptotic stability of synchronization manifold $\A$ of
system~\eqref{eqn:hc}.
\end{proof}
\\

Recall that system~\eqref{eqn:hc} is obtained from
system~\eqref{eqn:h} by a time-varying change of coordinates that
is a rotation in $\Real^{2}$. Since rotation is a norm-preserving
operation and set $\A$ is invariant under rotations,
Theorem~\ref{thm:Ac} yields the below result.

\begin{theorem}\label{thm:A}
Consider harmonic oscillators~\eqref{eqn:h}. Synchronization
manifold $\A$ is semiglobally practically asymptotically stable.
\end{theorem}

\section{Conclusion}\label{sec:conclude}

We have shown that nonlinearly-coupled Lienard-type oscillators
(almost) synchronize provided that their phases initially lie in a
semicircle and the frequency of oscillations is high enough. We
have generated the same result for nonlinearly-coupled harmonic
oscillators without any requirement on their initial phases. We
have employed averaging techniques to establish our main theorems.

\bibliographystyle{plain}
\bibliography{references}

\end{document}

%% file: macros.tex
\usepackage[final]{graphicx}
\usepackage{psfrag}
\usepackage{amsmath,amsfonts,amssymb,amsxtra,subeqnarray}

\newcommand {\Kinf}{\ensuremath{\mathcal{K}_\infty}}

\newcommand {\K}{\ensuremath{\mathcal{K}}}
\newcommand {\Real}{\ensuremath{{\mathbb{R}}}}

\newcommand{\C}{\ensuremath{\mathcal C}}
\newcommand{\R}{\ensuremath{\mathcal R}}

\newcommand{\D}{\ensuremath{\mathcal D}}

\newcommand{\A}{\ensuremath{\mathcal A}}
\newcommand{\I}{\ensuremath{\mathcal I}}

\newcommand{\setS}{\ensuremath{\mathcal S}}
\newcommand{\setH}{\ensuremath{\mathcal H}}

\newcommand{\setE}{\ensuremath{\mathcal E}}

\newcommand{\G}{\ensuremath{\mathcal G}}
\newcommand{\N}{\ensuremath{\mathcal N}}
\newcommand{\W}{\ensuremath{\mathcal W}}

\newcommand{\B}{\ensuremath{\mathcal B}}

\renewcommand{\ss}{\scriptscriptstyle}

\newcommand{\ellc}{\ensuremath{\ell^{\ss \circlearrowleft}}}
\newcommand{\ellcbar}{\ensuremath{\bar\ell^{\ss \circlearrowleft}}}
\newcommand{\hc}{\ensuremath{h^{\ss \circlearrowleft}}}
\newcommand{\hcbar}{\ensuremath{\bar h^{\ss \circlearrowleft}}}

\newcommand{\co}{{\mbox{\rm co}}}

\newtheorem{theorem}{Theorem}

\newtheorem{lemma}{Lemma}
\newtheorem{claim}{Claim}

\newtheorem{definition}{Definition}
\newtheorem{remark}{Remark}

\newenvironment{proof}{\noindent {\bf Proof.}}{\hfill \hspace*{1pt}\hfill$\blacksquare$}